\documentclass{amsart}[12pt]
\usepackage{amscd,amsmath,amsfonts,amssymb,enumerate,amsthm}
\usepackage{graphicx}
\usepackage{wasysym}
\usepackage{ifpdf}
\usepackage{enumitem}
\usepackage[hyperfootnotes=false]{hyperref}
\hypersetup{hidelinks}
\usepackage{setspace}
\usepackage[usenames]{xcolor}
\usepackage{amsrefs}
\usepackage[utf8]{inputenc}
\usepackage[english]{babel}
\usepackage{fullpage}

\newtheorem{thm*}{Theorem}[]
\newtheorem{prop}{Proposition}[section]
\newtheorem{thm}[prop]{Theorem}
\newtheorem{cor}[prop]{Corollary}

\newtheorem{rem}[prop]{Remark}

\newtheorem{lem}[prop]{Lemma}
\newcommand{\mref}[1]{(\ref{#1})}

\title[Behavior of the complete Kähler-Einstein metric in tube domains]{Asymptotic behavior of the complete Kähler-Einstein metric in some convex domains}
\author{Sebastien Gontard}

\email{Sebastien.Gontard@univ-grenoble-alpes.fr}

\address{Univ. Grenoble Alpes, CNRS, IF, 38000 Grenoble, France}



\begin{document}
\maketitle
\begin{abstract}
We study the complete Kähler-Einstein metric in tube domains $\{z\in \mathbb{C}^2 / \; Re(4pz_1) + Re(z_2)^{2p} <1 \}$ where $p\in \mathbb{N}^\ast$. We obtain estimates of this metric and its holomorphic bisectional curvatures near the weakly pseudoconvex boundary points and use these estimates to obtain non tangential estimates for the holomorphic bisectional curvatures of the Kähler-Einstein metric in some convex domains.
\end{abstract}

\subsection*{Introduction}
In 1980, S.-Y. Cheng and S.-T. Yau proved that every bounded strictly pseudoconvex  domain $\Omega \subset \mathbb{C}^n$, $n\geq 2$, with boundary of class $\mathcal{C}^7$, admits a complete Kähler-Einstein metric of negative Ricci curvature. Namely, they proved that there exists a (unique) strictly plurisubharmonic solution $g\in \mathcal{C}^{\omega}\left(\Omega\right)$ to the problem:
\begin{eqnarray}
\label{MAC}
Det\left[g_{i\bar{j}}\right]&=e^{(n+1)g} \quad \text{on $\Omega$},
 \\
\label{BVal}
g&=+\infty \quad \text{on $\partial \Omega$}
\end{eqnarray}
(see \cite{CY}). By comparing this solution to the approximate solutions constructed by C. Fefferman in \cite{Fef2}, they proved that if $\Omega$ is bounded, strictly pseudoconvex with boundary of class $\mathcal{C}^{\max\left(2n+9,3n+6\right)}$, then $e^{-g}\in \mathcal{C}^{n+1+\delta}\left(\overline{\Omega}\right)$ for every $\displaystyle \delta \in \left[0,\frac{1}{2}\right[$, and the holomorphic sectional curvatures of this metric tend to -2 at any boundary point of $\Omega$, which are the curvatures of the unit ball equipped with its Bergman-Einstein metric.
\\The existence of a complete Kähler-Einstein metric of negative Ricci curvature has been extended to the case of bounded domains of holomorphy by N.Mok and S.-T.Yau in \cite{MY}. Moreover, J.Bland proved that if $q\in \partial \Omega$ is a "nice" strictly pseudoconvex boundary point of $\partial \Omega$, then there exists a neighborhood $U$ of $q$ such that $e^{-g}\in \mathcal{C}^{n+1+\delta}\left(\overline{\Omega \cap U}\right)$ (see \cite{Bla1}). We proved that this result still holds without the "nice" condition, and also proved that the holomorphic sectional curvatures of the metric tend to -2 at $q$ (see \cite{Gon1}).
\\Although the regularity of $e^{-g}$ and the behavior of the holomorphic bisectional curvatures at weakly pseudoconvex boundary points is unknown in general, J.Bland proved that in the Thullen domains $\{(z,w)\in \mathbb{C}^{n-1}\times \mathbb{C} / \lvert z \rvert ^2 + \lvert w \rvert ^{2p}<1\}$, the sectional curvatures are bounded from above and below by negative constants if $p\geq 1$. He also proved some estimates for the metric and its volume form (see \cite{Bla2}).
\\In this paper we adapt the method used by J.Bland in \cite{Bla2} to prove curvatures estimates in the tube domains $T_p:=\{z\in \mathbb{C}^2 / \; Re(4pz_1) + Re(z_2)^{2p} <1 \}$ where $p\in \mathbb{N}^\ast$. More precisely, we prove:

\begin{thm*}
\label{Main}
There exist positive constants $0<c\leq C$ and $0<\alpha <1$ such that we have the following for every $z\in T_p\cap \left(\left\lbrace \frac{Re(z_2)^{2p}}{1-Re(4pz_1)} \leq \alpha \right\rbrace \cup \left\lbrace   1-\alpha \leq \frac{Re(z_2)^{2p}}{1-Re(4pz_1)}<1 \right\rbrace \right)$:
\[\forall v,w\in \mathbb{C}^2\setminus\{0\},\quad -C\leq Bis_z(v,w) \leq -c ,\]
where $Bis_{z}(v,w)$ stands for the holomorphic bisectional curvature of the complete Kähler-Einstein metric induced by the potential $g$ satisfying conditions \mref{MAC} and \mref{BVal} on $T_p$, at point $z$ and at vectors $v$ and $w$.
\end{thm*}


The regions $\left\lbrace \frac{Re(z_2)^{2p}}{1-Re(4pz_1)} \leq \alpha \right\rbrace$ and $\left\lbrace   1-\alpha \leq \frac{Re(z_2)^{2p}}{1-Re(4pz_1)}<1 \right\rbrace $ are related to the geometry of the orbits of $T_p$ under the action of its automorphism group, and replace the usual conical regions (see Section 2 for details).

The Bergman metric is another interesting metric induced by a Kähler potential. S.Fu gave estimates of the Bergman metric and its holomorphic sectional curvatures on the axis $\mathbb{R}\times \{0\} + i\mathbb{R}^2$ in tube domains $T_p$. He deduced that for every smoothly bounded complete Reinhardt domain of finite type there exists a neighborhood of the boundary on which holomorphic sectional curvatures for the Bergman metric are pinched between negative constants (see \cite{Fu}). Following the method used in the present paper, the computations from \cite{Fu}, the asymptotic expansion of the Bergman kernel in tube domains obtained by J.Kamimoto in \cite{Kam1}, a remark made by K.-T. Kim and S. Lee about the Bergman-Fuchs formula (see Section 6 in \cite{KL1}), we notice that the conclusion of Theorem \ref{Main} holds for the Bergman metric. However, the behavior of the holomorphic sectional curvatures of the Bergman metric away from the axis $\mathbb{R}\times \{0\} + i\mathbb{R}^2$ is still unknown.

Using a rescaling technique, results of J.Bland in Thullen domains (see \cite{Bla2}) and the results of the present paper, we prove:

\begin{thm*}
\label{Conv}
Let $D\subset \mathbb{C}^2$ be a bounded convex domain with boundary of class $\mathcal{C}^\infty$. Let $q\in \partial D$ be a point of finite type $2p$ such that a local model at $q$ is either a Thullen domain or a tube domain. Let $\left(z^{(\nu)}\right)_{\nu \in \mathbb{N}}\in D^\mathbb{N}$ be a sequence that converges non-tangentially to $q$. Then, there exists positive constants $0<c\leq C$ and an integer $\nu_0\in \mathbb{N}$ such that:
\[-C\leq \sup_{\stackrel{v,w \in \mathbb{C}^2 \setminus\{0\}}{\nu \geq \nu_0}}Bis_{z^{(\nu)}}\left(v,w\right) \leq -c.\]
\end{thm*}

\vspace{3mm}

This paper is organized as follows. In Section 1 we introduce some notations that will be used in the paper.

In Section 2 we describe the boundary of the domain $T_p$, its automorphism group $Aut(T_p)$, and introduce a parametrization of the orbits of $T_p$ under the natural action of its automorphism group. We also give a relationship between the regions introduced in Theorem \ref{Main} and the notions of non-tangential and tangential convergences.

In Section 3 we use the automorphism group of the domain to obtain an equation satisfied by the restriction of the Kähler-Einstein potential to the subset $\{0\}\times ]-1,1[$, and we analyze this equation to deduce the asymptotic behavior of the Kähler-Einstein potential.

In Section 4 we use the results proved in Section 3 to prove Theorem \ref{Main} (see Theorems \ref{BisNorm} and \ref{BisHTG}). The proofs of Theorems \ref{BisNorm} and \ref{BisHTG} give precise estimates of the constants $c,C$ depending on the approach region.

Finally in Section 5 we prove Theorem \ref{Conv}.

\section{Notations}
For the rest of this paper, we fix an integer $p\in \mathbb{N}^\ast$.
\\We work with the topology induced by the Euclidean norm given by $\left \lvert z \right \rvert ^2:=\left \lvert z_1\right \rvert ^2 +\left \lvert z_2 \right \rvert ^2$ for every number $z=(z_1,z_2) \in \mathbb{C}^2$. We denote by $S(0,1):=\{z\in \mathbb{C}^2 / \left \lvert z \right \rvert =1\}$ the unit sphere centered at the origin.

Let $I\subset \mathbb{R}$ be an open interval and $k\in \mathbb{N}$. We denote by $\mathcal{C}^k\left(I\right)$ the set of real-valued functions that are $k$-times differentiable in $I$, and we denote by $\mathcal{C}^\omega\left(I\right)$ the set of analytic functions in $I$. If $f\in \mathcal{C}^k\left(I\right)\cup \mathcal{C}^\omega\left(I\right)$ and $0\leq j \leq k$ is an integer, we denote by $f^{(j)}$ the $j-th$ derivative of $f$.

Let $U\subset \mathbb{C}^2$ be a domain (that is a connected open subset) and $k\in \mathbb{N}$. We denote by $\mathcal{C}^k\left(U\right)$ the set of real valued functions that are $k$-times differentiable in $U$, and we denote by $\mathcal{C}^\omega\left(U\right)$ the set of real valued analytic functions in $U$. 
If $f\in \mathcal{C}^1\left(U\right)$ and $1\leq j \leq 2$ is an integer, we use the standard notations $f_{j}:=\frac{1}{2}\left(\frac{\partial f}{\partial x_j}-i\frac{\partial f}{\partial y_j}\right)$ and $f_{\bar{j}}:=\frac{1}{2}\left(\frac{\partial f}{\partial x_j}+i\frac{\partial f}{\partial y_j}\right)$.

Let $U,V \subset \mathbb{C}^2$ be two domains and $k\in \mathbb{N}$. We denote by $\mathcal{C}^k\left(U,V\right)$ the set of functions $f=\left(f_1,f_2\right)$ having values in $V$ such that the coordinate functions $f_1$ and $f_2$ are $k$-times differentiable in $U$, and we denote by $\mathcal{C}^\omega\left(U,V\right)$ the set of analytic functions in $U$ having values in $V$. 
\\If $f\in \mathcal{C}^1\left(U,V\right)$, we denote by $Jac_\mathbb{C}(f):=\left[\frac{\partial f_i}{\partial z_j}\right]$ its complex Jacobian. It is a matrix whose entries are continuous functions in $U$.

\section{Description of the domain and its automorphism group}
In this section we describe the boundary of the domain $T_p$ and its automorphism group.
Recall that $T_p$ is a tube in the sense that $T_p=B_p+i \mathbb{R}^2,$ with
$B_p:=\{(x_1,x_2) \in \mathbb{R}^2 / 4px_1+x_2^{2p}<1\}.$
The following proposition describes the boundary points of $T_p$: 

\begin{prop}
\label{PCVBord}
Let $z\in \partial T_p$. If $z \in (\frac{1}{4p},0)+i\mathbb{R}^2$, $\partial T_p$ is weakly pseudoconvex at $z$. Otherwise, $\partial T_p$ is strictly pseudoconvex at $z$. Moreover, every weakly pseudoconvex boundary point of $\partial T_p$ is of finite type $2p$ in the sense of D'Angelo.
\end{prop}

We recall the description of the automorphism group of the tube, denoted by $Aut(T_p)$:
\begin{prop}
\label{Generateurs}
The automorphism group $Aut(T_p)$ of $T_p$ is spanned by the following affine maps: 
\begin{itemize}[leftmargin=*]
\item 
Translations of vectors with real coordinates: $\tau_u(z_1,z_2)=(z_1,z_2)+iu,$ where $u\in \mathbb{R}^2$,
\item 
Dilations: $ d_\lambda(z_1,z_2)=\left(\frac{\lambda(4pz_1-1)+1}{4p},\lambda ^{\frac{1}{2p}}z_2\right)$, where $\lambda >0$,
\item
The symmetry of complex axis $\{z_2=0\}$: $s(z_1,z_2)=(z_1,-z_2)$.
\end{itemize}
\end{prop}


The translations have a Jacobian equal to the identity matrix. Also,
$Jac_\mathbb{C}(d_\lambda)=\left[\begin{matrix} &\lambda &0 \\&0 &\lambda^{\frac{1}{2p}} \end{matrix} \right]$ and $Jac_\mathbb{C}(s)=\left[\begin{matrix} &1 &0 \\&0 &-1 \end{matrix} \right]$.
\\For $j=1,2$, we denote by $\pi_j^\mathbb{R}$ the following map:
$$\begin{array}{ccll}
\pi_j^\mathbb{R} : &\mathbb{C}^2 &\longrightarrow &\mathbb{R}
\\ &(z_1,z_2) &\longmapsto &Re(z_j),
\end{array}.$$
Let $X:=\frac{\pi_2^\mathbb{R}}{(1-4p\pi_1^\mathbb{R})^{\frac{1}{2p}}}$. This function is well defined on the set $\{ z\in \mathbb{C}^2 / Re(4pz_1)<1\}$ which contains $T_p$. 

Moreover, observe that it satisfies the following properties:
\begin{itemize}[leftmargin=*]
\item $X\in \mathcal{C}^\infty \left(\{ z\in \mathbb{C}^2 / Re(4pz_1)<1\}\right)$,
\item $X$ is a parametrization of the orbits of $T_p$ under the action of $Aut(T_p)$, in the sense that
$$\forall F\in Aut(T_p),\quad \forall z \in T_p,\quad X(F(z))=X(z) \text{ and $X_{\lvert \{0\}\times ]-1,1[}$ is injective} ,$$
\item $X(T_p)=]-1,1[$,
\item $q\in \{\left \lvert X \right \rvert=1\}$ if and only if $q$ is a strictly pseudoconvex boundary point of $\partial T_p$.
\end{itemize}

Let us relate the regions introduced in Theorem \ref{Main} to the notions of tangential and non-tangential convergences. Let $\theta \in \left]0,\frac{\pi}{2}\right[$. We denote by $\Lambda(\theta):=\left\lbrace z \in T_p / \frac{\sqrt{Im(z_1)^2 +\left \lvert z_2 \right \rvert ^2}}{\left(\frac{1}{4p} -Re(z_1)\right)}\leq tan(\theta) \right \rbrace$ the half cone of vertex $\left(\frac{1}{4p},0\right)$, of axis $\mathbb{R}\times \{0\}$ and of angle $\theta$.

Let $(z^{(n)})_{n\in \mathbb{N}}\in T_p^\mathbb{N}$ such that $z^{(n)}\underset{n\to +\infty}{\longrightarrow}(\frac{1}{4p},0)$. Recall that $(z^{(n)})_{n\in \mathbb{N}}$ converges non-tangentially to $(\frac{1}{4p},0)$ if there exists a constant $\theta \in \left] 0,\frac{\pi}{2}\right [$ and an integer $N\in \mathbb{N}$ such that for every $n\geq N$ we have $z^{(n)}\in \Lambda(\theta)$, and that $(z^{(n)})_{n\in \mathbb{N}}$ converges tangentially to $(\frac{1}{4p},0)$ if for every constant $\theta \in \left] 0,\frac{\pi}{2}\right [$ there exists an integer $N\in \mathbb{N}$ such that for every integer $n\geq N$ we have $z^{(n)}\notin \Lambda(\theta)$.
\\Now observe that we have:
\begin{align*}
\forall z\in T_p,\quad 4p \left\lvert X(z)\right \rvert (1-Re(4pz_1))^{\frac{1}{2p}-1} \leq \frac{\sqrt{Im(z_1)^2 +\left \lvert z_2 \right \rvert ^2}}{\left(\frac{1}{4p} -Re(z_1)\right)},
\end{align*}
hence we deduce that for every sequence $(z^{(n)})_{n\in \mathbb{N}}\in T_p^\mathbb{N}$ that converges to $(\frac{1}{4p},0)$ and for every $0<\alpha <1$ we have:
\begin{align*}
\\ \forall 0<\theta<\frac{\pi}{2},\left(z^{(n)}\right)\in \Lambda(\theta)^\mathbb{N} \Rightarrow \exists N\in \mathbb{N},\quad \forall n\geq N,\quad z^{(n)}\in \{\lvert X\rvert \leq \alpha \},
\\ \left(z^{(n)}\right)\in \left \lbrace 1-\alpha \leq \left \lvert X \right \rvert <1\right \rangle \}^\mathbb{N} \Rightarrow \forall 0<\theta < \frac{\pi}{2},\quad \exists N\in \mathbb{N},\quad \forall n\geq N,\quad z^{(n)}\notin \Lambda(\theta).
\end{align*}
In particular, Theorem 1 gives the non-tangential behavior of the bisectional curvatures at weakly pseudoconvex boundary points of $T_p$, and also gives a "hyper-tangential" behavior of the bisectional curvatures at weakly pseudoconvex boundary points of $T_p$.

We conclude this section with the following Proposition, which directly follows from Proposition \ref{Generateurs} and the definition of $X$:
\begin{prop}
\label{TransfOrb}
Let $z \in T_p$ and define $\psi^{(z)}:=d_{\frac{1}{1-Re(4pz_1)}}\circ\tau_{-(Im(z_1),Im(z_2))}$. Then the map $\psi^{(z)}$ is an automorphism of $T_p$ which sends $z$ to $(0,X(z))$, and satisfies $Det\left(Jac_\mathbb{C}\left(\psi^{(z)}\right)\right)=\frac{1}{\left(1-Re(4pz_1)\right)^\frac{2p+1}{2p}}$.
\end{prop}

Proposition \ref{TransfOrb} enables to reduce the study of the metric and its curvatures on $T_p$ to the study of the same quantities on the set $\{0\}\times ]-1,1[$ (see Sections 3 and 4 for more details).

\section{Complete Kähler-Einstein metric on $T_p$}
Since $T_p$ is a pseudoconvex tube domain such that $B_p$ is convex and does not contain any line, it follows from the work of S.-Y. Cheng and S.-Y. Yau (see \cite{CY}) and N. Mok and S.-Y. Yau (see \cite{MY}) (see also Remark 3 at page 235 of \cite{Isa1} for more details, and also \cite{XY} for a more detailed study) that there exists a unique strictly plurisubharmonic function $g\in \mathcal{C}^\omega\left(T_p\right)$ satisfying conditions \mref{MAC} and \mref{BVal} on $\Omega=T_p$, and consequently the manifold $T_p$ equipped with the metric $\left[g_{i\bar{j}}\right]$ is a complete Kahler-Einstein manifold with Ricci curvature equals to $-3$.

According to results obtained in \cite{Gon1}, we know for every number that there exists a neighborhood $U \subset \mathbb{C}^2$ of $(0,1)$ such that for every number $0\leq \delta <\frac{1}{2}$ we have $e^{-g}\in \mathcal{C}^{3+\delta}\left(\overline{T_p\cap U} \right)$. In particular, the boundary behavior of the metric, of its volume, and  of its curvatures at strictly pseudoconvex boundary points are already known, contrary to their boundary behavior at weakly pseudoconvex points.

In the sequel, we study the potential $g$ and its behavior at the weakly pseudoconvex boundary point $\left(\frac{1}{4p},0\right)$. 

\begin{rem}
Let $\alpha \geq 0$. It also follows from the work of A.Isaev (see \cite{Isa1}) ang N. Xiang and X.-P. Yang (see \cite{XY}) that if the domain $\{Re(2\alpha z_1)+Re(z_2)^{\alpha}<1\}$ possesses a complete Kähler-Einstein metric, that is if there exists a solution to Equation \mref{MAC} with boundary condition \mref{BVal} on $\{Re(2\alpha z_1)+Re(z_2)^{\alpha}<1\}$, then $\alpha \in 2\mathbb{N}^\ast$.
\end{rem}

\subsection{The invariance property and an expression of the Kähler-Einstein potential in terms of a special auxiliary function}

The invariance property of the Kähler-Einstein metric under the action of $Aut(T_p)$ enables to simplify the expression of the Kähler-Eisntein potential $g$: 
\begin{prop}
\label{Form}
Let
$$\begin{array}{ccll}
F: &]-1,1[ &\longrightarrow &\mathbb{R}
\\&x &\longmapsto &g(0,x),
\end{array}$$ 
and set $K:=\frac{2p+1}{3}$. Then the following holds on $T_p$: 
\begin{equation}
\label{FormEq}
g=F\circ X +\frac{K}{p}Log\left(\frac{1}{1-4p\pi_1^\mathbb{R}}\right).
\end{equation}
\end{prop}

\begin{proof}[Proof of Proposition \ref{Form}]
The Kähler-Einstein metric is invariant under the action of $Aut(T_p)$, which means that:
\begin{equation}
\label{InvMetr}
\forall \psi\in Aut(T_p),\quad \left[g_{i\bar{j}}\right]=Jac_\mathbb{C}\left(\psi\right)^T\left[g_{i\bar{j}}\circ\psi \right]\overline{Jac_\mathbb{C}\left(\psi\right)}.
\end{equation}
We apply the function $Log\circ Det$ on both sides of Equation \mref{InvMetr} and use the Monge-Ampère Equation \mref{MAC} to deduce the following transformation formula:
\begin{equation}
\label{JacEq}
\forall \psi=(\psi_1,\psi_2)\in Aut(T_p),\quad g=g\circ \psi+\frac{2}{3}Log \left\lvert  Det(Jac_\mathbb{C}\left(\psi\right)\right \rvert .
\end{equation}
Let $z \in T_p$. We apply identity \mref{JacEq} to the function $\psi=\psi^{(z)}$ given in Proposition \ref{TransfOrb} and obtain the result.
\end{proof}

The function $F$ inherits from the Kähler-Einstein potential $g$ some regularity properties.
\begin{prop}
\label{Basicsf}
The function $F$ is real analytic on $]-1,1[$, strictly convex, and even. Moreover, $e^{-F}\in \mathcal{C}^{3+\delta}\left(\left[-1,1 \right]\right)$ for every number $\delta \in \left[0,\frac{1}{2}\right[$.
\end{prop}

\begin{proof}[Proof of Proposition \ref{Basicsf}] The relation $F(x)=g(0,x)$ for every number $x\in ]-1,1[$ directly implies that $F\in \mathcal{C}^\omega\left(]-1,1[\right)$ and $e^{-F}\in \mathcal{C}^{3+\delta}\left(\left[-1,1 \right]\right)$ for every number $\delta \in \left[0,\frac{1}{2}\right[$. In particular, by differentiating Equation \mref{FormEq} twice at the point $\left(0,x\right) \in T_p$, we obtain $F^{(2)}(x)=4g_{2\bar{2}}(0,x)>0$ because $g$ is strictly plurisubharmonic on $T_p$. Hence $F$ is strictly convex on $]-1,1[$. To prove that $F$ is even on $]-1,1[$, we use the automorphism $s$ introduced in Proposition \ref{Generateurs} to deduce that for every number $-1<x<1$, we have $F(x)=g(0,x)=F\left(X(0,-x)\right)+\frac{K}{p}Log(1)=F(-x)$, hence the result.
\end{proof}


\subsection{The Kähler-Einstein condition and two differential equations satisfied by $F$}

$\newline$
We use Equation \mref{FormEq} and the Monge-Ampère Equation \mref{MAC} to obtain a first differential equation satisfied by the function $F$:
\begin{prop}
\label{Thm1}
Denote $f:=F^{(1)}$. Then the metric $\left[g_{i\bar{j}}\right]$ satisfies the following on $T_p$: 
\begin{align}
\label{FormeExplicite}
\left[g_{i\bar{j}}\right]
 &=\displaystyle \left[ \begin{matrix}
\displaystyle &\frac{X^2f^{(1)}\circ X +(2p+1)Xf\circ X +4pK}{(1-4p\pi_1^\mathbb{R})^2} &\frac{X f^{(1)}\circ X +f\circ X}{2(1-4p\pi_1^\mathbb{R})^{1+\frac{1}{2p}}}
\\
\\&\frac{X f^{(1)}\circ X +f\circ X}{2(1-4p\pi_1^\mathbb{R})^{1+\frac{1}{2p}}} &\frac{f^{(1)}\circ X}{4\left(1-4p\pi_1^\mathbb{R}\right)^\frac{1}{p}}
\end{matrix}
\right],
\end{align}
\begin{equation}
\label{Det Z Eq}
\quad Det(g_{i \bar{j}})=\frac{Z(X)}{(1-4p\pi_1^\mathbb{R})^{\frac{3K}{p}}},
\end{equation}
where the function $Z$ is defined by $\displaystyle Z(x):=\frac{f^{(1)}(x)\left((2p-1)xf(x)+4pK\right)-f(x)^2}{4}$ for every number $x\in ]-1,1[$, and satisfies the following equation:
\begin{equation}
\label{EqZ}
Z=e^{3F} \text{ on $]-1,1[$}.
\end{equation}
\end{prop}

\begin{proof}[Proof of  Proposition \ref{Thm1}]
On $T_p$ we have:
$$\left[X_i\right]=\left[X_{\bar{i}}\right]=\left[
\begin{matrix}
\frac{X}{1-4p\pi_1^\mathbb{R}}
\\ \frac{1}{2(1-4p\pi_1^\mathbb{R})^{\frac{1}{2p}}}
\end{matrix}
\right],$$
$$\left[X_i X_{\bar{j}}\right]=
\left[
\begin{matrix}
\frac{X^2}{(1-4p\pi_1^\mathbb{R})^{2}} & \frac{X}{2(1-4p\pi_1^\mathbb{R})^{\frac{1}{2p}+1}}
\\ \frac{X}{2(1-4p\pi_1^\mathbb{R})^{\frac{1}{2p}+1}} &\frac{1}{4(1-4p\pi_1^\mathbb{R})^{\frac{1}{p}}}
\end{matrix}
\right],$$
$$\left[X_{i\bar{j}}\right]=\left[
\begin{matrix}
\frac{(2p+1)X}{(1-4p\pi_1^\mathbb{R})^{2}} &\frac{1}{2(1-4p\pi_1^\mathbb{R})^{\frac{1}{2p}+1}}
\vspace{2mm}
\\ \frac{1}{2(1-4p\pi_1^\mathbb{R})^{\frac{1}{2p}+1}} &0
\end{matrix}
\right].$$
Differentiating Equation \mref{FormEq}, we directly deduce: 
\begin{align*}
\left[g_{i\bar{j}}\right]&=f\circ X \left[X_{i\bar{j}}\right]+ f^{(1)} \circ X \left[X_{i}X_{\bar{j}}\right]+\frac{4Kp}{(1-4p\pi_1^\mathbb{R})^2}E_{11},
\\&=\displaystyle \left[ \begin{matrix}
\displaystyle &\frac{X^2f^{(1)}\circ X +(2p+1)f\circ X +4pK}{(1-4p\pi_1^\mathbb{R})^2} &\frac{X f^{(1)}\circ X +f\circ X}{2(1-4p\pi_1^\mathbb{R})^{1+\frac{1}{2p}}}
\\
\\&\frac{X f^{(1)}\circ X +f\circ X}{2(1-4p\pi_1^\mathbb{R})^{1+\frac{1}{2p}}} &\frac{f^{(1)}\circ X}{4\left(1-4p\pi_1^\mathbb{R}\right)^\frac{1}{p}}
\end{matrix}
\right].
\end{align*}
Then we apply the function on Det both sides of relation \mref{FormeExplicite} to directly Equation \mref{Det Z Eq}.
Finally, recall that according to Equations \mref{MAC} and \mref{FormEq} one has on $T_p$:
\begin{align*}
Det\left[g_{i\bar{j}}\right]&=e^{3g},
\\&=e^{3F\circ X - \frac{3K}{p}Log \left(1-4p\pi_1^\mathbb{R}\right)},
\\&=\frac{e^{3F \circ X}}{\left(1-4p\pi_1^\mathbb{R}\right)^{2+\frac{1}{p}}},
\end{align*}
hence Equation \mref{EqZ}.
\end{proof}

We use Equation \mref{EqZ} to obtain a differential equation satisfied by $f$ and $f^{(1)}$:
\begin{prop}
\label{Eqf}
The function $f$ satisfies the following equation for every $x\in]-1,1[$: 
\begin{align}
\label{Eqf Eq}
\left((2p-1)xf(x)+4pK\right)f^{(1)}(x)&=(2p-1)xf(x)^3 +\left(6pK+1\right)f(x)^2-2(p+1)\int_0^x f(t)^3 \;dt+4e^{3F(0)}.
\end{align}
\end{prop}

\begin{proof}[Proof of Proposition \ref{Eqf}]
Let $x\in ]-1,1[$. We multiply both sides of Equation \mref{EqZ} by $12f$, use the defintion of the function $Z$ and integrate from $0$ to $x$ to obtain: 
\begin{align*}
&4(3f(x)e^{3F(x)})=3(2p-1)xf(x)^2f^{(1)}(x)+12pKf(x)f^{(1)}(x)-3f(x)^3,
\\&4Z(x)=4e^{3F(x)}=3(2p-1)\int_0^x tf(t)^2f^{(1)}(t) \;dt + 6pKf(x)^2-3\int_0^x f(t)^3 \;dt +4e^{3F(0)}.
\end{align*}
We integrate by part the first term of the right hand side: 
\begin{align*}
\int_0^x tf(t)^2f^{(1)}(t) \;dt&=\left[\frac{tf(t)^3}{3}\right]_0^x-\frac{1}{3}\int_0^x f(t)^3 \;dt=\frac{xf(x)^3}{3}-\frac{1}{3}\int_0^x f(t)^3 \;dt.
\end{align*}
Using again the definition of $Z$ we obtain: 
\begin{align*}
&\left((2p-1)xf(x)+4pK\right)f^{(1)}(x)-f(x)^2=(2p-1)xf(x)^3 +6pKf(x)^2-2(p+1)\int_0^x f(t)^3 \;dt+4e^{3F(0)},
\\&\left((2p-1)xf(x)+4pK\right)f^{(1)}(x)=(2p-1)xf(x)^3+\left(6pK+1\right)f(x)^2-2(p+1)\int_0^x f(t)^3 \;dt+4e^{3F(0)}. \qedhere
\end{align*}
\end{proof}

\subsection{Asymptotic analysis of the auxiliary function}

$\newline$
In this subsection we use condition \mref{BVal}, Proposition \ref{Basicsf} and Equation \mref{Eqf Eq} to study the function $F$ and its derivatives. Since $F$ is an even function, we may restrict its study to the set $[0,1[$.
\\We point out that Propositions \ref{Asymptf0}, \ref{Asymptf1}, Corollary \ref{AsymptZ} and part of Proposition \ref{Asymptf2} may also be deduced from the work done in \cite{Gon1} because the function $F$ is the restriction of the Kähler-Einstein potential $g$ to the set $\{0\}\times ]-1,1[$, and $\partial T_p$ is smooth and strictly pseudoconvex at $(0,1)$. In this paper, we use Equation \mref{Eqf Eq} and the interior regularity of $F$ to derive these results.
\\From the strict convexity of $F$ and condition \mref{BVal} we have the following: 

\begin{prop}
\label{Asymptf0}
Every derivative of $F$ is unbounded in a neighborhood of $1^-$. Moreover, $f(x) \underset{x\to 1^-}{\longrightarrow} + \infty$.
\end{prop}

\begin{proof}[Proof of Proposition \ref{Asymptf0}]If there existed an integer $k\in \mathbb{N}$ such that $F^{(k)}$ was bounded in a neighborhood of $1^-$, then $g(0,\cdot)$ would be bounded in a neighborhood of $1^-$, which contradicts the hypothesis \mref{BVal}, hence every derivative of $F$ is unbounded in a neighborhood of $1^-$.
\\Since $F$ is a strictly convex, even function in $]-1,1[$, its derivative $f$ is a non negative increasing function on $[0,1[$. Since the function $f$ is not bounded bounded on $[0,1[$, we directly deduce that $f(x) \underset{x\to 1^-}{\longrightarrow} + \infty$.
\end{proof}

We use the following lemma to deduce the asymptotic behavior of $f^{(1)}$ at $x=1^-$: 
\begin{lem}
\label{LemCvx}
Let $F\in \mathcal{C}^1(]0;1[)$ be a convex function satisfying $\displaystyle \lim _{y\to 1^-} F^{(1)}(y)=+\infty$. Then: $\displaystyle \lim_{y\to 1^-} \frac{F(y)}{F^{(1)}(y)}=0.$
\end{lem}

\begin{proof}[Proof of Lemma \ref{LemCvx}]
Since $F$ satisfies $\displaystyle \lim _{y\to 1^-} F^{(1)}(y)=+\infty$, there exists a constant $a\in ]0,1[$ such that $F^{(1)}>0$ on $]a,1[$. Let $a<x<y<1$. Then $F^{(1)}(y)>0$ and $F(x)\leq F(y)$. We apply the fundamental theorem of calculus to the function $F$ to obtain the following:
\begin{align*}
0\leq F(y)-F(x)=\int_x^y F^{(1)}(t)\; dt \leq (y-x)F^{(1)}(y),
\end{align*}
so that $\frac{F(x)}{F^{(1)}(y)}\leq \frac{F(y)}{F^{(1)}(y)}\leq (y-x)+\frac{F(x)}{F^{(1)}(y)}$. Hence we deduce:
$$ \forall x\in ]0;1[,\quad 0 \leq \limsup_{y\to 1^-}\frac{F(y)}{F^{(1)}(y)}\leq 1-x,$$
therefore we obtain $\displaystyle \lim_{y\to 1^-} \frac{F(y)}{F^{(1)}(y)}=0$ by letting $x$ tend to $1^-$, hence the result.
\end{proof}

We use Equation \mref{Eqf Eq}, Proposition \ref{Asymptf0} and Lemma \ref{LemCvx} to obtain the first order asymptotic of $f$ at $x=1^-$:

\begin{prop}
\label{Asymptf1}
We have: $\displaystyle \lim_{x \to 1^-}\frac{f^{(1)}}{f^2}(x)=\lim_{ x\to  1^-}f(x)(1-x)=1$.
\end{prop}

\begin{proof}[Proof of Proposition \ref{Asymptf1}]
Let $x> 0$. Since $f(0)=0$ and $f$ is increasing on $[0,1[$, we have $f(x)> 0$. We divide Equation \mref{Eqf Eq} both sides by $f(x)$ to obtain the following: 
\begin{equation}
\label{Divf3Eq}
\left((2p-1)x+\frac{4pK}{f(x)}\right)\frac{f^{(1)}(x)}{f(x)^2}=(2p-1)x
+\frac{6pK+1}{f(x)}-2(p+1)\frac{\displaystyle \int_0^x f(t)^3 \;dt}{f(x)^3}+\frac{4e^{3F(0)}}{f(x)^3}.
\end{equation}
Let us prove that $\displaystyle \lim_{x\to 1^-} \frac{\displaystyle\int_0^x f(t)^3 \;dt}{f(x)^3} =0$. Define $\tilde{f}(x):=\displaystyle \int_0^x f(t)^3\; dt$ for $x\in [0,1[$. Then $\tilde{f} \in \mathcal{C}^1\left(]0,1[\right)$, is convex and satisfies $\displaystyle \lim_{x \to 1^-}\tilde{f}^{(1)}(x)= +\infty$. We apply Lemma \ref{LemCvx} to $\tilde{f}$ to deduce that $\displaystyle \lim_{x \to 1^-}\frac{\tilde{f}}{\tilde{f}^{(1)}}(x)=0$.
\\Define $\displaystyle b(x):=\frac{6pK+1}{f(x)}-2(p+1)\frac{\displaystyle \int_0^x f(t)^3 \;dt}{f(x)^3}+\frac{4e^{3F(0)}}{f(x)^3}$ for $x \in [0,1]$. Then $b\in \mathcal{C}\left([0,1]\right)$ and $\displaystyle \lim_{x\to 1^-} b(x)=0$. Hence $B:= \displaystyle \int_\cdot ^1b(t)\; dt$ is well defined and $B\in \mathcal{C}^1\left([0,1]\right)$. Let $x\in ]0,1[$. We integrate Equation \mref{Divf3Eq} between $x$ and $1$ to obtain: 
\begin{align*}
\displaystyle (2p-1)\frac{x}{f(x)}+(2p-1)\int_x^1\frac{dt}{f(t)}+\frac{2pK}{f(x)^2}&=\displaystyle \int_x^1\left((2p-1)t+\frac{4pK}{f(t)}\right)\frac{f^{(1)}}{f^2}(t)\; dt,
\\&= \frac{2p-1}{2}(1-x^2) +B(x),
\end{align*}
\begin{align}
\label{IntTempEq}
(2p-1)\left(1+\frac{2pK}{f(x)}\right)\frac{x}{f(x)(1-x)}+(2p-1)\frac{\displaystyle \int_x^1\frac{dt}{f(t)}}{1-x}&=\frac{2p-1}{2}(1+x)+\frac{B(x)}{1-x}.
\end{align}
Note that $\displaystyle \int_\cdot^1 \frac{dt}{f(t)}$ is the primitive of the function $\displaystyle \frac{1}{f}\in \mathcal{C}\left(]0,1]\right)$, so that $\displaystyle \lim_{x\to 1^-} \frac{\displaystyle \int_x^1\frac{dt}{f(t)}}{1-x}=0$. Likewise by construction of $B$ we have $\displaystyle \lim_{x\to 1^-}\frac{B(x)}{1-x}=0$. We let $x$ tend to $1^-$ in Equation \mref{IntTempEq} to deduce $\displaystyle\lim_{x\to 1^-} \frac{2p-1}{f(x)(1-x)}=\lim_{x\to 1^-} (2p-1)\left(1+\frac{2pK}{f(x)}\right)\frac{x}{f(x)(1-x)}=2p-1$, hence $\displaystyle\lim_{x\to 1^-} f(x)(1-x)=1$.
\end{proof}

Proposition \ref{Asymptf1} directly gives the asymptotic behaviour of $Z$, and also an asymptotic expansion of $F$: 
\begin{cor}
\label{AsymptZ}
We have: $\displaystyle \lim_{x \to 1^-}(1-x)^3Z(x)=\frac{2p-1}{4}$, and $\displaystyle F(x)=Log\left(\frac{1}{1-x}\right)+\frac{Log\left(\frac{2p-1}{4}\right)}{3}+\displaystyle \underset{x\to 1^-}{o\left(1\right)}$.
\end{cor}

\begin{proof}[Proof of Corollary \ref{AsymptZ}]
Proposition \ref{Asymptf1} and the definition of the function $Z$ directly gives the first result. We apply formula \mref{Det Z Eq} to deduce that $\displaystyle \lim_{x\to 1^-}\left((1-x)e^{F(x)}\right)^3=\frac{2p-1}{4}$, hence $\displaystyle \lim_{x\to 1^-} F(x)-Log\left(\frac{1}{1-x}\right)=\frac{Log\left(\frac{2p-1}{4}\right)}{3}$, hence the second result.
\end{proof}

We directly deduce from Corollary \ref{AsymptZ} the asymptotic behavior of the potential $g$, or equivalently its volume form (see Proposition \ref{Thm1}). In order to estimate the curvatures of the Kähler-Einstein metric, we also need the asymptotic behavior of the derivatives of higher order of $F$ at $x=1^-$. We have the following: 

\begin{prop}
\label{Asymptf2}
For every integer $k\in \mathbb{N}$, one has: 
\[ \lim_{x\to 1^-} (1-x)^{k+3}Z^{(k)}(x)=\frac{2p-1}{8}(k+2)!\; \text{and} \; \lim_{x\to 1^-}f^{(k)}(x)(1-x)^{k+1}=\lim_{x\to 1^-}\frac{f^{(k)}}{f^{k+1}}(x)=k! \quad .\]
\end{prop}

\begin{proof}[Proof of Proposition \ref{Asymptf2}]
The fact that $\displaystyle \lim_{x\to 1^-}\frac{f^{(k)}}{f^{k+1}}(x)=k!$ directly follows from Proposition \ref{Asymptf1} and $\displaystyle \lim_{x\to 1^-}f^{(k)}(x)(1-x)^{k+1}=k!$. Let us prove the other assertions by induction. Proposition \ref{Asymptf1} and Corollary \ref{AsymptZ} ensure that the formulas are true for $k=0$. Let $k\geq 0$ be an integer and assume that the formulas are true for any integer $ 0\leq l \leq k$. We differentiate Equation \mref{EqZ} $(k+1)$ times to obtain  $Z^{(k+1)}=\left(Z^{(1)}\right)^{(k)}=3\sum_{l=0}^k \binom{k}{l}f^{(l)}Z^{(k-l)}$, hence the following:
\begin{align*}
\\ \lim_{x\to 1^-}(1-x)^{k+4}Z^{(k+1)}(x)&=3\sum_{l=0}^k \binom{k}{l}\lim_{x\to 1^-}\left((1-x)^{l+1}f^{(l)}(x)\right)\lim_{x\to 1^-}\left((1-x)^{k-l+3}Z^{(k-l)}(x)\right),
\\& =3\frac{2p-1}{8}\sum_{l=0}^k \binom{k}{l}l!(k+2-l)!,
\\&=3\frac{2p-1}{8}k!\sum_{l=0}^k (k+2-l)(k+1-l),
\\&=3\frac{2p-1}{8}k!\sum_{l=1}^{k+1} l(l+1)=\frac{2p-1}{8}(k+3)!.
\end{align*}
We differentiate Equation \mref{EqZ} $k$ times to obtain: 
\begin{align*}
\frac{(k+2)!}{2}=\lim_{x\to 1^-}4(1-x)^{k+3}\frac{Z^{(k)}}{2p-1}(x)&=\lim_{x\to 1^-}\sum_{l=0}^{k}\binom{k}{l} \left((1-x)^{l+2}f^{(l+1)}(x)\right)\left((1-x)^{k+1-l}f^{(k-l)}(x)\right),
\\&=\sum_{l=0}^{k-1} \binom{k}{l}(l+1)!(k-l)!+\lim_{x\to 1^-}\left((1-x)^{k+2}f^{(k+1)}(x)\right),
\\&=k!\sum_{l=0}^{k-1} (l+1)+\lim_{x\to 1^-}\left((1-x)^{k+2}f^{(k+1)}(x)\right),
\\&=\frac{k(k+1)!}{2}+\lim_{x\to 1^-}\left((1-x)^{k+2}f^{(k+1)}(x)\right),
\end{align*}
\begin{align*}
\lim_{x\to 1^-}\left((1-x)^{k+2}f^{(k+1)}(x)\right)=\frac{(k+2)!}{2}-\frac{k(k+1)!}{2}=(k+1)!,
\end{align*}
hence the result.
\end{proof}

\begin{rem}
\begin{itemize}[leftmargin=*]
\item We conjecture that $e^{-F}$ satisfies the following:
\begin{equation}
\label{ConjEq}
\exists (\eta_k)_{k\in \mathbb{N}}\in \mathcal{C}^\infty\left([0,1]\right)^\mathbb{N},\quad e^{-F(x)}\underset{x \to 1^-}{\sim}(1-x)\sum_{k=0}^{+\infty} \eta_k\left((1-x)^3Log(1-x)\right)^k,
\end{equation}
with $\displaystyle \lim_{x\to 1^-}\eta_1(x)\neq 0$ except for $p=1$. Especially, apart from the case of the ball ($p=1$), one would have $e^{-F}\notin \mathcal{C}^4\left([0,1]\right)$ so that the regularity given in Proposition \ref{Basicsf} is almost optimal.
\\Conjecture \mref{ConjEq} is motivated by results of J.Lee and R.Melrose (see \cite{LMbebo}) and of R. Graham (see \cite{Gra1}) in the case of smooth strictly pseudoconvex domains, and by J.Kamimoto  (see \cite{Kam1}) in the case of the Bergman metric in tube domains.
\item In a forthcoming note, we will prove that in the case of the Thullen domains $\{\left \lvert z_1 \right \rvert ^2 + \left \lvert z_2 \right \rvert ^{2p}<1\}\subset \mathbb{C}^2$, there exists a positive function $\eta \in \mathcal{C}^\infty\left([0,1]\right)$ such that for every $x\in [0,1]$ we have $ e^{-g(0,x)} =(1-x)\eta(x)$ (so that in the sense of conjecture \mref{ConjEq} all the functions $\eta_k's$ are equal to $0$ for every integer $k\geq 1$), and that if we denote by $K$ the Bergman kernel of $\{\left \lvert z_1 \right \rvert ^2 + \left \lvert z_2 \right \rvert ^{2p}<1\}$ we have $x\mapsto K(0,x)e^{-3g(0,x)} \in \mathcal{C}^\infty \left([-1,1]\right)$, and deduce from this a comparison of the Kähler-Einstein metric and its curvatures to the Bergman metric and its curvatures, globally on the Thullen domain.
\end{itemize}
\end{rem}

\section{Curvatures estimates}
In this Section we use the analysis of the function $f$ obtained in Section 2 to get estimates of the holomorphic bisectional curvatures of the Kähler-Einstein metric.

\subsection{General properties of the holomorphic bisectional curvatures in $T_p$}

We wish to estimate the holomorphic (bi)sectional curvatures of the Kähler-Einstein metric $g$. If $v,w \in \mathbb{C}^2\setminus\{0\}$, and $z \in T_p$, we denote by $Bis_z(v,w)$ the holomorphic bisectional curvature of the metric $g$ at point $z$ between the vectors $v$ and $w$, and we denote by $ S_z(v)=Bis_z(v,v)$ the holomorphic sectional curvature of $g$ at point $z$ and at vector $v$. We shall omit the point at which we compute it and use the notations $Bis(v,w), S(v)$ to simplify the notations when possible. 

Recall that the holomorphic bisectional curvature of $g$ between non zero vectors $v$ and $w$ is given by:

\begin{equation}
\label{BisEq}
Bis(v,w)=\displaystyle\frac{\displaystyle \sum_{1\leq i,j,k,l \leq 2}R_{i \bar{j} k \bar{l}}v_i \bar{v_j}w_k \bar{w_l}}{\left(\sum_{1\leq i,j\leq 2}g_{i\bar{j}}v_i\overline{v_j}\right)\left(\sum_{1\leq i,j\leq 2}g_{i\bar{j}}w_i\overline{w_j}\right)},
\end{equation}
and that the curvature coefficients satisfy the following:
\begin{equation}
\label{CurvCoefEq}
\forall 1\leq i,j,k,l\leq 2, \quad R_{i \bar{j} k \bar{l}}=-g_{i \bar{j} k \bar{l}}+\sum_{1\leq \alpha,\beta \leq 2}g_{i k \bar{\alpha}}g^{\bar{\alpha}\beta}g_{\beta \bar{j}\bar{l}},
\end{equation}
where $\left[ g^{\bar{\alpha}\beta}\right]:=\left[g_{i \bar{j}}\right]^{-1}$. Recall that the holomorphic bisectional curvature of a metric does not depend on the length of the vectors at which it is computed, namely it satisfies the following:

\begin{equation}
\label{InvResc}
\forall v,w \in \mathbb{C}^2\setminus\{0\},\quad Bis(v,w)=Bis\left(\frac{v}{\lvert v \rvert },\frac{w}{\lvert w \rvert } \right).
\end{equation}

Moreover, since the Kähler-Einstein metric is invariant under the action of $Aut(T_p)$, its holomorphic bisectional curvatures satisfies the following transformation formula:

\begin{equation}\label{TransForm}
\forall z\in T_p,\quad \forall \psi\in Aut(T_p),\quad \forall v,w \in \mathbb{C}^2\setminus\{0\},\quad Bis_{\psi(z)}(\partial \psi_z(v),\partial \psi_z(w))=Bis_z(v,w).
\end{equation}

The following formula is specific to the tube domains: 
\begin{prop}\label{CurvFormTub}
If $v=(v_1,v_2)$, $w=(w_1,w_2)$, $\alpha$ is an argument of $v_1\overline{v_2}$ and $\beta$ is an argument of $w_1\overline{w_2}$, then the following holds: 
\begin{equation}\label{CurvFormTubEq}
\begin{array}{clll}
\left(\sum_{1\leq i,j\leq 2}g_{i\bar{j}}v_i\overline{v_j}\right)\left(\sum_{1\leq i,j\leq 2}g_{i\bar{j}}w_i\overline{w_j}\right)Bis(v,w)&=R_{1\bar{1}1\bar{1}}|v_1|^2|w_1|^2
\\&+2R_{1\bar{1}1\bar{2}}|v_1||w_1|(|v_1||w_2|cos(\beta)+|v_2||w_1|cos(\alpha))
\\&+R_{1\bar{1}2\bar{2}}(|v_1|^2|w_2|^2+|v_2|^2|w_1|^2+2|v_1||v_2||w_1||w_2|cos(\alpha-\beta))
\\&+2R_{1\bar{2}1\bar{2}}|v_1||v_2||w_1||w_2|cos(\alpha+\beta)
\\&+2R_{1\bar{2}2\bar{2}}|v_2||w_2|(|v_1||w_2|cos(\alpha)+|v_2||w_1|cos(\beta))
\\&+R_{2\bar{2}2\bar{2}}|v_2|^2|w_2|^2.
\end{array}
\end{equation}
\end{prop}

\begin{proof}[Proof of Proposition \ref{CurvFormTub}]
From the expression of the curvature coefficients \mref{CurvCoefEq} and the fact that $g$ and all its complex derivatives are real numbers, we have for $1\leq i,j,k,l\leq 2$: $R_{i \bar{j} k \bar{l}}=R_{ k \bar{j}i \bar{l}}=R_{j \bar{i}  l \bar{k}}$. Hence we may simplify formula \mref{BisEq} by gathering the terms depending on the number of $2$ occuring in the 4-uple $(i,j,k,l)$: 

\begin{align*}
\left(\sum_{1\leq i,j\leq 2}g_{i\bar{j}}v_i\overline{v_j}\right)\left(\sum_{1\leq i,j\leq 2}g_{i\bar{j}}w_i\overline{w_j}\right)Bis(v,w)&=R_{1 \bar{1}1 \bar{1}}|v_1|^2|w_1|^2
\\ &+R_{1 \bar{1} 1 \bar{2}}\left(|v_1|^2(w_1 \overline{w_2}+\overline{w_1}w_2 )+(v_1 \overline{v_2}+\overline{v_1}v_2 )|w_1|^2\right)
\\ &+R_{1 \bar{1} 2 \bar{2}}\left(|v_1|^2|w_2|^2+|v_2|^2|w_1|^2+v_1\overline{v_2}\overline{w_1}w_2+\overline{v_1}v_2w_1\overline{w_2}\right)
\\ &+R_{1 \bar{2} 1 \bar{2}}\left(v_1\overline{v_2}w_1\overline{w_2}+\overline{v_1}v_2\overline{w_1}w_2\right)
\\ &+R_{1 \bar{2} 2 \bar{2}}\left((v_1 \overline{v_2}+\overline{v_1}v_2) |w_2|^2+ |v_2|^2(w_1 \overline{w_2}+\overline{w_1}w_2)\right)
\\ &+R_{2 \bar{2} 2 \bar{2}}|v_2|^2|w_2|^2,
\\
\\&=R_{1\bar{1}1\bar{1}}|v_1|^2|w_1|^2
\\&+2R_{1\bar{1}1\bar{2}}|v_1||w_1|(|v_1||w_2|cos(\beta)+|v_2||w_1|cos(\alpha))
\\&+R_{1\bar{1}2\bar{2}}(|v_1|^2|w_2|^2+|v_2|^2|w_1|^2+2|v_1||v_2||w_1||w_2|cos(\alpha-\beta))
\\&+2R_{1\bar{2}1\bar{2}}|v_1||v_2||w_1||w_2|cos(\alpha+\beta)
\\&+2R_{1\bar{2}2\bar{2}}|v_2||w_2|(|v_1||w_2|cos(\alpha)+|v_2||w_1|cos(\beta))
\\&+R_{2\bar{2}2\bar{2}}|v_2|^2|w_2|^2.
\end{align*}
\end{proof}

Observe that Proposition \ref{TransForm} reduces the study of the behavior of the curvatures on $T_p$ to the study of the behavior of the curvatures on the subset $\{0\}\times [-1,1]$.

\subsection{Holomorphic bisectional curvatures when $X \longrightarrow 0$}
First we compute the curvature coefficients at the origin in the following: 
\begin{prop}\label{CoefOri}
The curvature coefficients satisfy the following at the origin: 
\begin{align*}
&R_{1 \bar{1}1 \bar{1}}=-32p^3K,
\\&R_{1 \bar{2}1\bar{2}}=(p-1)f^{(1)}(0),
\\&R_{1 \bar{1} 2 \bar{2}}=-pf^{(1)}(0),
\\&R_{2 \bar{2}2 \bar{2}}=\left(-3+\frac{1}{K}\right)\frac{f^{(1)}(0)^2}{16},
\end{align*}
all the other coefficients being equal to $0$.
\end{prop}

\begin{proof}[Proof of Proposition \ref{CoefOri}]Recall that $F$ is an even function, hence $f=F^{(1)}=g_{2}(0,\cdot)$ is an odd function. From this we directly deduce that, at $z=0$: 
 \begin{align*}
 &g_{12}=g_{21}=0,
 \\ &g_{1 1 2}=g_{1 2 1}=g_{2 1 1}=0,
 \\ &g_{2 2 2}=0,
 \\ &g_{1112}=g_{1121}=g_{1211}=g_{2111}=0,
 \\ &g_{1222}=g_{2122}=g_{2212}=g_{2221}=0.
 \end{align*}
Hence the coefficients in Equation \mref{CurvCoefEq} simplify into: 
\begin{align*}
&R_{1 \bar{1} 1 \bar{1}}=-g_{1 \bar{1} 1 \bar{1}} + g_{11 \bar{1}}g^{\bar{1}1}g_{1 \bar{1}\bar{1}},
\\&R_{1 \bar{2} 1 \bar{2}}=-g_{1 \bar{2} 1 \bar{2}}+g_{1 1 \bar{1}}g^{\bar{1}1} g_{1 \bar{2}\bar{2}},
\\&R_{1 \bar{1} 2 \bar{2}}=-g_{1 \bar{1} 2 \bar{2}}+g_{1 2 \bar{2}}g^{\bar{2}2}g_{2 \bar{1}\bar{2}},
\\ &R_{2 \bar{2} 2 \bar{2}}=-g_{2 \bar{2}2 \bar{2}}+g_{2 2\bar{1}}g^{\bar{1}1}g_{1 \bar{2} \bar{2}},
\\&R_{1 \bar{1} 1 \bar{2}}=R_{1 \bar{2} 2 \bar{2}}=0.
\end{align*}
We use formula \mref{FormeExplicite} to compute the derivatives of $g$ at the origin. We have, at $z=0$: 
\begin{align*}
\\&\left[g_{i\bar{j}}\right]=
\left[ \begin{matrix}
\displaystyle &4pK &0
\\
\\&0 &\frac{f^{(1)}(0)}{4}
\end{matrix}
\right],\left[ g^{\bar{\alpha}\beta}\right]=
\left[ \begin{matrix}
\displaystyle &\frac{1}{4pK} &0
\\
\\&0 &\frac{4}{f^{(1)}(0)}
\end{matrix}
\right],
\end{align*}
\[g_{111}=16p^2K, \;
g_{122}=\frac{f^{(1)}(0)}{2}, \;
g_{1111}=96p^3K, \;
g_{1122}=(p+1)f^{(1)}(0), \;
g_{2222}=\frac{f^{(3)}(0)}{16}.\]
Thus we obtain: 
\begin{align*}
&R_{1 \bar{1}1 \bar{1}}=-32p^3K,
\\&R_{1 \bar{2}1\bar{2}}=(p-1)f^{(1)}(0),
\\&R_{1 \bar{1} 2 \bar{2}}=-pf^{(1)}(0),
\\&R_{2 \bar{2}2 \bar{2}}=-\frac{f^{(3)}(0)}{16}+\frac{f^{(1)}(0)^2}{16pK}.
\end{align*}
According to Equation \mref{EqZ}, we have $Z^{(2)}(0)=3f^{(1)}(0)Z(0),$ that is $4pKf^{(3)}(0)+4(p-1)f^{(1)}(0)^2=12pKf^{(1)}(0)^2$, hence $R_{2 \bar{2}2 \bar{2}}=\left(-3+\frac{1}{K}\right)\frac{f^{(1)}(0)^2}{16}$.
\end{proof}

From the computations of Proposition \ref{CoefOri} we deduce the precise upper and lower bounds for the holomorphic bisectional curvatures and holomorphic sectional curvatures at the origin: 
\begin{prop}\label{PincNorm}
Let $v,w \in \mathbb{C}^2\setminus\{0\}$. Then we have:
$$ -3+\frac{3}{2p+1}\leq Bis_{0}(v,w)\leq -\frac{3}{2p+1} \quad \text{and} \quad S_{0}(v)\leq -\frac{3}{2}-\frac{1}{2pK}.$$
Moreover,
$$ Bis_{0}((1,0)(1,0))=-3+\frac{3}{2p+1}, \quad Bis_{0}((1,0)(0,1))= -\frac{3}{2p+1},$$
$$\text{and } S_{0}\left(\left(\frac{1}{\sqrt{4pK}},\frac{\sqrt{f^{(1)}(0)}}{2}\right)\right)=-\frac{3}{2}-\frac{1}{2pK}.$$
\end{prop}

\begin{proof}[Proof of Proposition \ref{PincNorm}]Let $C:=-3+\frac{3}{2p+1}$. Using Proposition \ref{CoefOri} we have:
\begin{align*}
&\frac{R_{1\bar{1}1\bar{1}}}{g_{1\bar{1}}^2}=-\frac{2p}{K}=C,
\\ &\frac{R_{2\bar{2}2\bar{2}}}{g_{2\bar{2}}^2}=-\frac{2p}{K}=C,
\\ &\frac{R_{1\bar{1}2\bar{2}}}{g_{1\bar{1}}g_{2\bar{2}}}=-\frac{1}{K}=\frac{C}{2p},
\\ &\frac{R_{1\bar{2}1\bar{2}}}{g_{1\bar{1}}g_{2\bar{2}}}=\frac{p-1}{pK}=-\frac{(p-1)C}{2p^2},
\end{align*}
hence:
\begin{align*}
\frac{1}{C}\left(\sum_{1\leq i,j\leq 2}g_{i\bar{j}}v_i\overline{v_j}\right)\left(\sum_{1\leq i,j\leq 2}g_{i\bar{j}}w_i\overline{w_j}\right)Bis(v,w)&=\underbrace{g_{1\bar{1}}^2|v_1|^2|w_1|^2+g_{1\bar{1}}g_{2\bar{2}}(|v_1|^2|w_2|^2+|v_2|^2|w_1|^2)+g_{2\bar{2}}^2|v_2|^2|w_2|^2}_{=\left(\sum_{1\leq i,j\leq 2}g_{i\bar{j}}v_i\overline{v_j}\right)\left(\sum_{1\leq i,j\leq 2}g_{i\bar{j}}w_i\overline{w_j}\right)}
\\&-\frac{2p-1}{2p}g_{1\bar{1}}g_{2\bar{2}}\displaystyle \bigg(|v_1|^2|w_2|^2+|v_2|^2|w_1|^2 
\\&\left.-\frac{2}{2p-1}|v_1||v_2||w_1||w_2|cos(\alpha-\beta)\right.
\\&+\left.\frac{2(p-1)}{p(2p-1)}|v_1||v_2||w_1||w_2|cos(\alpha+\beta)\right),
\end{align*}
\begin{align*}
-\frac{2p}{2p-1}\left(\frac{1}{C}Bis(v,w)-1\right)\left(\sum_{1\leq i,j\leq 2}g_{i\bar{j}}v_i\overline{v_j}\right)\left(\sum_{1\leq i,j\leq 2}g_{i\bar{j}}w_i\overline{w_j}\right)&=g_{1\bar{1}}g_{2\bar{2}}\bigg(|v_1|^2|w_2|^2+|v_2|^2|w_1|^2 
\\&-\frac{2}{2p-1}|v_1||v_2||w_1||w_2|cos(\alpha-\beta)
\\&\left.+\frac{2(p-1)}{p(2p-1)}|v_1||v_2||w_1||w _2|cos(\alpha+\beta)\right).
\end{align*}

Observe that the lower bound of the right-hand side is achieved if $cos(\alpha-\beta)=-cos(\alpha+\beta)=1$, and the upper bound is achieved if $cos(\alpha-\beta)=-cos(\alpha+\beta)=-1$, so that we only need to investigate two cases to obtain the bounds for $Bis(v,w)$. Note that in the case of the holomorphic sectional curvatures, we have $\alpha=\beta$, so that the condition $cos(\alpha-\beta)=-cos(\alpha+\beta)=-1$ cannot occur, hence we have to study this case differently.

\begin{itemize}[leftmargin=*]
\item Case 1: $cos(\alpha-\beta)=-cos(\alpha+\beta)=1$. Computations yield to:
\begin{align*}
&g_{1\bar{1}}g_{2\bar{2}}\left(|v_1|^2|w_2|^2+|v_2|^2|w_1|^2
-\frac{2}{2p-1}|v_1||v_2||w_1||w_2|
-\frac{2(p-1)}{p(2p-1)}|v_1||v_2||w_1||w_2|\right)
\\&=g_{1\bar{1}}g_{2\bar{2}}\left(|v_1|^2|w_2|^2+|v_2|^2|w_1|^2-\frac{2}{p}|v_1||v_2||w_1||w_2|\right)
\\&\geq g_{1\bar{1}}g_{2\bar{2}}\left(|v_1|^2|w_2|^2+|v_2|^2|w_1|^2-2|v_1||v_2||w_1||w_2|\right)
\\&= g_{1\bar{1}}g_{2\bar{2}}(|v_1||w_2|-|v_2||w_1|)^2\geq 0,
\end{align*}
and equality holds if $|v_1||w_2|=|v_2||w_1|$, for instance if $v_1=w_1=0$. Hence $\displaystyle C=-3+\frac{3}{2p+1}\leq Bis_{0}(v,w)$ for every vectors $v,w\in \mathbb{C}^2\setminus\{0\}$.
\item Case 2: $cos(\alpha-\beta)=-cos(\alpha+\beta)=-1$. Computations yield to:
\begin{align*}
&g_{1\bar{1}}g_{2\bar{2}}\left(|v_1|^2|w_2|^2+|v_2|^2|w_1|^2+\frac{2}{p}|v_1||v_2||w_1||w_2|\right)
\\&=|v|_g^2|w|_g^2-g_{1\bar{1}}^2|v_1|^2|w_1|^2-g_{2\bar{2}}^2|v_2|^2|w_2|^2+\frac{2}{p}g_{1\bar{1}}g_{2\bar{2}}|v_1||v_2||w_1||w_2|
\\&\leq|v|_g^2|w|_g^2-g_{1\bar{1}}^2|v_1|^2|w_1|^2-g_{2\bar{2}}^2|v_2|^2|w_2|^2+2g_{1\bar{1}}g_{2\bar{2}}|v_1||v_2||w_1||w_2|
\\&=|v|_g^2|w|_g^2-(g_{11}|v_1||w_1|-g_{22}|v_2||w_2|)^2
\\&\leq |v|_g^2|w|_g^2,
\end{align*}
and equality holds for instance if $u_1=v_2=0$. Hence $\displaystyle Bis_{0}(v,w)\leq -\frac{3}{2p+1}$ for every vectors $v,w\in \mathbb{C}^2\setminus\{0\}$.
\item Case 3: $v=w$. We have:
\begin{align*}
-\frac{2p}{2p-1}\left(\frac{1}{C}S(v)-1\right)\left(\sum_{1\leq i,j\leq 2}g_{i\bar{j}}v_i\overline{v_j}\right)^2&=2g_{1\bar{1}}g_{2\bar{2}}|v_1|^2|v_2|^2\left(1-\frac{1}{2p-1}+\frac{(p-1)}{p(2p-1)}cos(2\alpha)\right)
\\ &\leq \left(1-\frac{1}{p(2p-1)}\right)\frac{\left(\sum_{1\leq i,j\leq 2}g_{i\bar{j}}v_i\overline{v_j}\right)^2}{2},
\\\frac{K}{2p}S(v)&\leq -1+\frac{2p-1}{4p}\left(1-\frac{1}{p(2p-1)}\right)=-\frac{3K}{4p}-\frac{1}{4p^2}
\\S(v)&\leq -\frac{3}{2}-\frac{1}{2pK},
\end{align*}
and equality holds if $g_{1\bar{1}}\left\lvert v_1\right \rvert^2=g_{2\bar{2}}\left\lvert v_2\right \rvert^2$. Hence $\displaystyle S_{0}(v)\leq -\frac{3}{2}-\frac{1}{2pK}$ for every vector $v\in \mathbb{C}^2\setminus\{0\}$. \qedhere
\end{itemize}
\end{proof}

\begin{rem}
It can be deduced from the computations done by J.Bland in \cite{Bla1} that we obtain the same upper and lower bounds for the curvatures of the Kähler-Einstein metric in the Thullen domains $\{\left\lvert z_1\right\rvert^2+\left\lvert z_2\right\rvert^{2p}<1\}$, although Thullen domains and tube domains are not biholomorphic except for $p=1$.
\end{rem}

We can deduce from Proposition \ref{PincNorm} part of Theorem \ref{Main}:
\begin{thm}
\label{BisNorm}
There exist positive constants $0<c\leq C$ and $\alpha>0$ such that 
\[\forall v,w\in \mathbb{C}^2\setminus\{0\},\quad \forall z\in\{\left\lvert X \right \rvert \leq \alpha\},\quad -C\leq Bis_z\left(v,w\right) \leq -c.\]
\end{thm}
\begin{proof}
Since the map
\[\begin{array}{ccll}
&\left]-1,1\right[\times S(0,1)^2 &\longrightarrow &\mathbb{R}
\\& \left(x,v,w\right) &\longmapsto &Bis_{(0,x)}(v,w)
\end{array}\]
is continuous, it is uniformly continuous on every subset of the form $J\times S(0,1)^2$ where $J\subset ]-1,1[$ is a compact set. Especially, we deduce that for every positive number $\epsilon>0$ there exists a positive constant $\alpha >0$ such we have the following:
\begin{equation}
\label{ContNorm}
\forall (x,v,w)\in [-\alpha, \alpha]\times S(0,1)^2,\quad \left \lvert Bis_{(0,x)}(v,w)-Bis_0(v,w) \right \rvert \leq \epsilon.
\end{equation}
Let $\epsilon \in \left] 0, \frac{3}{2p+1}\right[$. Let $\alpha >0$ be such that \mref{ContNorm} holds. We use relations \ref{TransForm} and \ref{InvResc} to obtain that for every point $z\in \{\lvert X \rvert \leq \alpha \}$ and every vectors $v,w\in S(0,1)^2$:
\begin{align*}
&\left \lvert Bis_z(v,w) -Bis_0(\partial \psi^{(z)}_z(v),\partial\psi^{(z)}_z(w))\right \rvert = \left \lvert Bis_{(0,X(z))}(\partial \psi^{(z)}_z(v),\partial\psi^{(z)}_z(w)) -Bis_0(\partial \psi^{(z)}_z(v),\partial\psi^{(z)}_z(w))\right \rvert ,
\\ &= \left \lvert Bis_{(0,X(z))}\left(\frac{\partial \psi^{(z)}_z(v)}{\left \lvert \partial \psi^{(z)}_z(v) \right \rvert},\frac{\partial \psi^{(z)}_z(w)}{\left \lvert \partial \psi^{(z)}_z(w) \right \rvert}\right) -Bis_0\left(\frac{\partial \psi^{(z)}_z(v)}{\left \lvert \partial \psi^{(z)}_z(v) \right \rvert},\frac{\partial \psi^{(z)}_z(w)}{\left \lvert \partial \psi^{(z)}_z(w) \right \rvert}\right)\right \rvert ,
\\&\leq \epsilon,
\end{align*}
therefore, using Proposition \ref{PincNorm} we deduce:
\begin{align*}
-3+\frac{3}{2p+1}-\epsilon \leq Bis_0(\partial \psi^{(z)}_z(v),\partial\psi^{(z)}_z(w)) \leq Bis_{z}\left(v,w\right) \leq Bis_0(\partial \psi^{(z)}_z(v),\partial\psi^{(z)}_z(w)) \leq -\frac{3}{2p+1}+\epsilon<0,
\end{align*}
hence the result.
\end{proof}

\subsection{Holomorphic bisectional curvatures when $\left \lvert X\right \rvert \longrightarrow 1 $}
In this subsection, we use the asymptotic behavior of $F$ obtained in Proposition \ref{Asymptf2} up to order 4 to prove the remaining part of Theorem \ref{Main}. It will follow from the computation of $\lim_{x \to 1^-} Bis_{(0,x)}(v,w)$: 

\begin{thm}\label{BisHTG}
There exist positive constants $0<c\leq C$ and $\alpha>0$ such that 
\[\forall v,w\in \mathbb{C}^2\setminus\{0\},\quad \forall z\in\{\left\lvert 1-\left\lvert X \right \rvert \right \rvert \leq \alpha\},\quad -C\leq Bis_z\left(v,w\right) \leq -c.\]
\end{thm}

\begin{proof}[Proof of Theorem \ref{BisHTG}]
Because of the invariance of $T_p$ under the symmetry $s$ introduced in Proposition \ref{Generateurs}, it is enough to prove that there exist positive constants $0<c\leq C$ and $0<\alpha <1$ such that: 
\[\forall v,w\in \mathbb{C}^2\setminus\{0\},\quad \forall z\in\{ 1-\alpha \leq X <1\},\quad -C\leq Bis_z\left(v,w\right) \leq -c.\]
First we prove the following:
\begin{equation}
 \label{BisHTGEq0}
 \lim_{x \to 1^-}  \sup_{v,w\in S(0,1)^2} \left(Bis_{(0,x)}(v,w)+1+\frac{\left\lvert \displaystyle \sum_{1\leq i,j\leq 2} g_{i\bar{j}}(0,x)v_i\overline{w_j}\right\rvert^2}{\left( \displaystyle \sum_{1\leq i,j\leq 2} g_{i\bar{j}}(0,x)v_i\overline{v_j}\right)\left( \displaystyle \sum_{1\leq i,j\leq 2} g_{i\bar{j}}(0,x)w_i\overline{w_j}\right)}\right)=0.
\end{equation}
The conclusion then follows from the invariance properties of the metric.
\\First we prove that $\frac{R_{i\bar{j}k\bar{l}}(0,x)}{f(x)^4}\underset{x\to 1^-}{\sim}-2\left(X_{i}X_{\bar{j}}X_k X_{\bar{l}}\right)(0,x)$ and $\frac{\left(  g_{i\bar{j}}g_{k\bar{l}}+g_{i\bar{l}}g_{k\bar{j}}\right)(0,x)}{f(x)^4}\underset{x\to 1^-}{\sim}2\left(X_{i}X_{\bar{j}}X_k X_{\bar{l}}\right)(0,x)$.
\\Let $1\leq i,j,k,l,\alpha,\beta \leq 2$. We differentiate relation \mref{FormeExplicite} to obtain:
\begin{align*}
&g_{i\bar{j}}=f^{(1)}X_{i}X_{\bar{j}}+fX_{i\bar{j}}+\frac{K}{p}Log\left(\frac{1}{1-4p\pi_1^\mathbb{R}}\right)_{i\bar{j}},
\\&Zg^{\bar{\alpha}\beta}=(-1)^{\alpha+\beta}\left(f^{(1)}X_{\overline{3-\alpha}}X_{3-\beta}+fX_{\overline{3-\alpha} 3-\beta}+\frac{K}{p}Log\left(\frac{1}{1-4p\pi_1^\mathbb{R}}\right)_{\overline{3-\alpha} 3-\beta}\right),
\\&g_{i\bar{j}k}=f^{(2)}X_{i}X_{\bar{j}}X_k+f^{(1)}\left(X_{i\bar{j}}X_k+X_{ik}X_{\bar{j}}+X_{k\bar{j}}X_i\right)+fX_{i\bar{j}k}+\frac{K}{p}Log\left(\frac{1}{1-4p\pi_1^\mathbb{R}}\right)_{i\bar{j}k},
\\&g_{i\bar{j}k\bar{l}}=f^{(3)}X_i X_{\bar{j}}X_k X_{\bar{l}}
\\&+f^{(2)}\left(X_{i\bar{j}}X_k X_{\bar{l}}+X_{ik}X_{\bar{j}} X_{\bar{l}}+X_{i\bar{l}}X_{\bar{j}}X_k +X_{k\bar{j}}X_i X_{\bar{l}}+X_{k\bar{l}}X_i X_{\bar{j}}+X_{\bar{j}\bar{l}}X_i X_k\right)
\\&+f^{(1)}\left(X_{i\bar{j}k}X_{\bar{l}}+X_{i\bar{j}\bar{l}}X_k+X_{ik\bar{l}}X_{\bar{j}}+X_{\bar{j}k\bar{l}}X_i +X_{i\bar{j}}X_{k\bar{l}}+X_{ik}X_{\bar{j}\bar{l}}+X_{i\bar{l}}X_{k\bar{j}}\right)
\\&+fX_{i\bar{j}k\bar{l}}+\frac{K}{p}Log\left(\frac{1}{1-Re(4pz_1)}\right)_{i\bar{j}k\bar{l}}.
\end{align*}
We directly deduce from the expression of $g_{i\bar{j}}$ and Proposition \ref{Asymptf2} that $g_{i\bar{j}}(0,x)\underset{x\to 1^-}{\sim}f(x)^2\left(X_iX_{\bar{j}}\right)(0,x)$, hence $\left(g_{i\bar{j}}g_{k\bar{l}}+g_{i\bar{l}}g_{k\bar{j}}\right)(0,x)\underset{x\to 1^-}{\sim}2f(x)^4\left(X_iX_{\bar{j}}X_kX_{\bar{l}}\right)(0,x)$. 
\\Likewise we obtain $g_{i\bar{j}k\bar{l}}(0,x)\underset{x\to 1^-}{\sim}6f(x)^4\left(X_i X_{\bar{j}}X_k X_{\bar{l}}\right)(0,x)$.
\\Moreover, in the expression $g_{i k \bar{\alpha}}g^{\bar{\alpha}\beta}g_{\beta \bar{j}\bar{l}}$, the contribution of a term of the form $(-1)^{\alpha}X_{\bar{\alpha}}X_{\overline{3-\alpha}}$ or $(-1)^{\beta} X_{\beta}X_{3-\beta}$ is 0. Using Proposition \ref{Asymptf2} we obtain:
\begin{align*}
 \left(\sum_{1\leq \alpha,\beta \leq 2}g_{i k \bar{\alpha}}g^{\bar{\alpha}\beta}g_{\beta \bar{j}\bar{l}}\right)(0,x)&\underset{x\to 1^-}{\sim}\left(\frac{\left(f^{(2)}\right)^2f}{Z}\right)(x)\left(X_iX_{\bar{j}}X_kX_{\bar{l}}\sum_{1\leq \alpha,\beta \leq 2}(-1)^{\alpha+\beta}X_{\overline{3-\alpha}3-\beta}X_{\bar{\alpha}}X_\beta\right)(0,x)
 \\&=\left(\frac{\left(f^{(2)}\right)^2f}{Z}\right)(x)\left(X_iX_{\bar{j}}X_kX_{\bar{l}}\right)(0,x)\frac{2p-1}{4}
 \\&
\\&\underset{x\to 1^-}{\sim}4f(x)^4\left(X_iX_{\bar{j}}X_kX_{\bar{l}}\right)(0,x).
\end{align*}
Therefore we deduce that $R_{i\bar{j}k\bar{l}}(0,x)\underset{x\to 1^-}{\sim}-2f(x)^4\left(X_iX_{\bar{j}}X_kX_{\bar{l}}\right)(0,x)$.
Hence we have at point $(0,x)$ for $0<x<1$:
\begin{align*}
&\sup_{v,w\in S(0,1)^2} Bis_{(0,x)}(v,w)
+1+\frac{\left\lvert \displaystyle \sum_{1\leq i,j\leq 2} g_{i\bar{j}}(0,x)v_i\overline{w_j}\right\rvert^2}{\left( \displaystyle \sum_{1\leq i,j\leq 2} g_{i\bar{j}}(0,x)v_i\overline{v_j}\right)\left( \displaystyle \sum_{1\leq i,j\leq 2} g_{i\bar{j}}(0,x)w_i\overline{w_j}\right)}
\\&=\sup_{v,w\in S(0,1)^2}\displaystyle \frac{\displaystyle \sum_{1\leq i,j,k,l\leq 2}\left(\displaystyle\frac{R_{i\bar{j}k\bar{l}}(0,x)}{f(x)^4}
+\displaystyle\frac{\left(g_{i\bar{j}}g_{k\bar{l}}+g_{i\bar{l}}g_{k\bar{j}}\right)(0,x)}{f(x)^4}\right)v_i\overline{v_j}w_k\overline{w_l}}{\left(\displaystyle\frac{  \sum_{1\leq i,j\leq 2} g_{i\bar{j}}(0,x)v_i\overline{v_j}}{f(x)^2}\right)\left( \displaystyle\frac{ \sum_{1\leq i,j\leq 2} g_{i\bar{j}}(0,x)w_i\overline{w_j}}{f(x)^2}\right)},
\\ &\underset{x\to 1^-}{\displaystyle \longrightarrow}0.
\end{align*}
This gives formula \mref{BisHTGEq0}. Since for every vectors $v,w\in \mathbb{C}^2\setminus\{0\}$ and every point $z\in T_p$ we have 
\[-2\leq -1-\frac{\left\lvert \displaystyle \sum_{1\leq i,j\leq 2} g_{i\bar{j}}(z)v_i\overline{w_j}\right\rvert^2}{\left( \displaystyle \sum_{1\leq i,j\leq 2} g_{i\bar{j}}(z)v_i\overline{v_j}\right)\left( \displaystyle \sum_{1\leq i,j\leq 2} g_{i\bar{j}}(z)w_i\overline{w_j}\right)}\leq -1,\]
we conclude exactly as in the end of the proof of Theorem \ref{BisNorm}.

\end{proof}

\section{Asymptotic behavior of the Kähler-Einstein metric in some convex domains}
We briefly recall some results about the Kobayashi metric and domains satisying a squeezing property which are needed in our proof of Theorem \ref{Conv}. 
\\We denote by $\Delta$ the plane unit disk.
For a domain $D\subset \mathbb{C}^n$, a point $z \in D$ and a vector $v\in \mathbb{C}^n \setminus\{0\}$ the Kobayashi pseudo-metric at point $z$ and vector $v$ is defined by :
\[K_D \left( z,v \right) :=\inf \left\lbrace \left\lvert \xi \right\rvert,\; \exists f:\Delta\mapsto D \text{ holomorphic satisfying }  f(0)=z \text{ and }\displaystyle \sum_{1\leq i \leq n}f'_i(z)\xi_i=v \right\rbrace.\]
We denote by $\left[g^{K,D}_{i\bar{j}}(z)\right]$ the matrix representative of the semi-definite metric obtained by polarizing the pseudo-quadratic form $v\mapsto K_D(z,v)$ in the canonical basis of $\mathbb{C}^n$. An important fact is that if $D$ is a convex domain that does not contain a complex line, then actually $\left[g^{K,D}_{i\bar{j}}\right]>0$ and the distance induced by the metric $\left[g^{K,D}_{i\bar{j}}\right]$ is complete on $D$.

The Hausdorff distance beteween two compact sets $A,B\subset \mathbb{C}^n$ is defined by: 
\[d_H\left(A,B\right):=max\left\lbrace \sup_{a\in A}\inf_{b\in B}\left\lvert a-b\right\rvert, \sup_{b\in B}\inf_{a\in A}\left\lvert b-a\right\rvert\right\rbrace.\]
The space of compact sets of $\mathbb{C}^n$ is complete for this distance. We say that a sequence of open convex domains $\left(\Omega_\nu\right)_{\nu \in \mathbb{N}}$ converges to an open convex domain $\Omega_\infty$ in the locall Hausdorff topology if there exists a number $R_0\geq 0$ such that for every number $R\geq R_0$ the sequence  $\left(\overline{\Omega_\nu \cap B(0,R)}\right)_{\nu \in \mathbb{N}}$ converges to $\overline{\Omega_\infty \cap B(0,R)}$ for the Hausdorff distance. 
\\We will use the following fact in the proof of Theorem \ref{Conv}: let $\left(\Omega_\nu\right)_{\nu \in \mathbb{N}}$ be a sequence of bounded convex sets with smooth boundary. Assume that $\left(\Omega_\nu\right)_{\nu \in \mathbb{N}}$ converges in the local Hausdorff topology of convex sets to a convex domain $\Omega_\infty$, and let $K\subset \Omega_\infty$ be a compact set. Then there exists an integer $\nu_K\in \mathbb{N}$ such that for every integer $\nu\geq\nu_K$ one has $K\subset \Omega_\nu$, and the sequence $\left(\left[g^{K,\Omega_\nu}_{i\bar{j}}\right]\right)_{\nu \geq \nu_K}$ converges uniformly to $\left[g^{K,\Omega_\infty}_{i\bar{j}}\right]$ on $K$. In that case we say that $\left(\left[g^{K,\Omega_\nu}_{i\bar{j}}\right]\right)_{\nu \in \mathbb{N}}$ converges uniformly on compact sets of $\Omega_\infty$ to $\left[g^{K,\Omega_\infty}_{i\bar{j}}\right]$. See Theorem 4.1 and Lemma 4.4 of \cite{Zim} for more details about these results.
\\

Let $D \subset \mathbb{C}^n$ be a domain and $a\in \left]0,1\right]$. We say that $D$ satisfies the $a$-squeezing property if for every point $z \in \Omega$, there exists a holomorphic injective map  $f\in 
\mathcal{H}\left(D,B\left(0,1\right)\right)$ satisfying $f(z)=0$ and $B\left(0,a\right)\subset f\left(D\right)$. We refer the reader to \cite{DGZ1}, \cite{KZ} and \cite{Yeu} for the study of domains satisfying a squeezing property.

In the following, for a domain $D\subset \mathbb{C}^2$, we use the notation $g^{E,D}$, respectively $Bis^D$ to denote the Kähler-Einstein potential of $D$ solution of Equation \mref{MAC} with boundary condition \mref{BVal}, respectively its holomorphic bisectional curvatures.

We can now prove Theorem \ref{Conv}:
\begin{proof}[Proof of Theorem \ref{Conv}]
We use a rescaling method to change the study of the boundary behavior  of the holomorphic bisectional curvatures into the study of the interior convergence for the sequence of the "rescaled" Kähler-Einstein metrics.

Our hypothesis on the local expression of $\partial D$ at $q$ implies that there exists an affine map $\psi \in Aut\left(\mathbb{C}^2\right)$ and a neighborhood $U$ of $q$ such that $\psi\left(q\right)=0$ and $\psi\left(D\cap U\right)=\left\lbrace Re(z_1)+H\left(z_2\right)+O\left(\left\lvert z_2 \right \rvert ^{2p+1}+\left\lvert z_1 \right \rvert \left\lvert z \right \rvert \right)<0\right\rbrace \cap \psi\left(U\right)$, with  either $\forall z \in \mathbb{C},\;H(z)=\left\lvert z\right\rvert ^{2p}$ or $\forall z \in \mathbb{C},\;H(z)=Re\left(z\right)^{2p}$. 
Since $\psi$ maps $D$ to $\psi\left(D\right)$ biholomorphically, we have the following by the invariance property of the Kähler-Einstein metric:
\[\forall z\in D,\quad \forall v,w\in \mathbb{C}^2\setminus\{0\},\quad Bis_{z}^D\left(v,w\right)=Bis_{\psi\left(z\right)}^{\psi\left(D\right)}\left(\partial \psi _{z}\left(v\right),\partial\psi _{z}\left(w\right)\right).\]
Moreover the sequence $\left(\psi\left(z^{(\nu)}\right)\right)_{\nu \in \mathbb{N}}$ converges non tangentially to $\psi\left(q\right)=0$ because invertible affine maps preserve cones. Thus up to replacing $D$ with $\psi\left(D\right)$ and $U$ with $\psi\left(U\right)$ we may assume that $q=0$ and $D\cap U=\left\lbrace Re(z_1)+H\left(z_2\right)+O\left(\left\lvert z_2 \right \rvert ^{2p+1}+\left\lvert z_1 \right \rvert \left\lvert z \right \rvert \right)<0\right\rbrace \cap U$. 
\\In this setting the condition of non-tangential convergence of $\left(z^{(\nu)}\right)_{\nu \in \mathbb{N}}\in D^\mathbb{N}$ means that $\displaystyle \left(\frac{-Re\left(z^{(\nu)}_1\right)}{\left\lvert z^{\nu}\right\rvert}\right)_{\nu \in \mathbb{N}}$ is bounded from below by a positive constant, thus up to taking a subsequence we may assume that $\left(\frac{z^{(\nu)}}{\left\lvert z^{(\nu)}\right\rvert}\right)_{\nu \in \mathbb{N}}$ converges to a point $z^{(\infty)}$ with $Re\left(z^{(\infty)}_1\right)<0$. Let 
\[\begin{array}{cccc}
\Lambda^{(\nu)}: &\mathbb{C}^2 &\longrightarrow &\mathbb{C}^2
\\ &z &\longmapsto &\displaystyle \left(\frac{z_1-z^{(\nu)}_1}{Re\left(-z^{(\infty)}_1\right)\left\lvert z^{(\nu)}\right\rvert},\frac{z_2-z^{(\nu)}_2}{\left(Re\left(-z^{(\infty)}_1\right)\left\lvert z^{(\nu)}\right\rvert\right)^{\frac{1}{2p}}}\right),
\end{array}\]
and set $\Omega_\nu:=\Lambda^{(\nu)}\left(D\right)$. As the image of the set $D$ by the affine map $\Lambda^{(\nu)}$, the set $\Omega_\nu$ is a bounded convex domain with boundary of class $\mathcal{C}^\infty$. From results in \cite{Ga}, $\left(\Omega_\nu\right)_{\nu \in \mathbb{N}}$ converges to the model $\Omega_\infty:=\left\lbrace Re(z_1)+H\left(z_2\right)<1\right\rbrace$ in the local Hausdorff topology. 
Then according to \cite{Zim} the sequence $\left(\left[g^{K,\Omega_\nu}_{i\bar{j}}\right]\right)_{\nu \in \mathbb{N}}$ converges uniformly to $\left[g^{K,\Omega_\infty}_{i\bar{j}}\right]$ on compact sets of $\Omega_\infty$.

Also, observe that $\Lambda^{(\nu)}\left(z^{(\nu)}\right)=0$, hence:
\begin{equation}
\label{TransfoDilaEq}
\forall z\in D,\quad \forall v,w\in \mathbb{C}^2\setminus\{0\},\quad Bis_{z^{(\nu)}}^D\left( \left(\partial\Lambda^{(\nu)}_{z^{(\nu)}}\right)^{-1}(v),\left(\partial\Lambda^{(\nu)}_{z^{(\nu)}}\right)^{-1}(w)\right)=Bis_{0}^{\Omega_\nu}\left(v,w\right),
\end{equation}
where the Kähler-Einstein potential $g^{E,\Omega_\nu}$ on $\Omega_\nu$ is induced by the Kähler-Einstein potential $g^{E,D}$. Assume momentarily that for every compact $K\subset \Omega_\infty$ the sequence $\left(g^{(\nu)}\right)_{\nu \geq \nu_K}$ converges to the Kähler-Einstein potential $g^{\Omega_\infty}$ of $\Omega_\infty$ in $\mathcal{C}^4\left(K\right)$. Then according to formula \mref{CurvCoefEq} and relation \mref{TransfoDilaEq} we deduce:
\[
\sup_{v,w\in\mathbb{C}^2\setminus\{0\}}\left\lvert Bis_{z^{(\nu)}}^D\left( \left(\partial\Lambda^{(\nu)}_{z^{(\nu)}}\right)^{-1}(v),\left(\partial\Lambda^{(\nu)}_{z^{(\nu)}}\right)^{-1}(w)\right)-Bis^{\Omega_\infty}_0\left(v,w\right)\right\rvert \underset{\nu \to \infty}{\longrightarrow} 0,
\]
so that we obtain Theorem \ref{Conv} using relation \mref{InvResc}, Theorem \ref{Main} and results in \cite{Bla2}. Thus it remains to prove that for every compact $K\subset \Omega_\infty$ the sequence $\left(g^{(\nu)}\right)_{\nu \geq \nu_K}$ converges to the Kähler-Einstein potential $g^{\Omega_\infty}$ of $\Omega_\infty$ in $\mathcal{C}^4\left(K\right)$. In fact we prove that for every integer $k\in \mathbb{N}$, the sequence $\left(g^{(\nu)}\right)_{\nu \geq \nu_K}$ converges to $g^{\Omega_\infty}$ in $\mathcal{C}^4\left(K\right)$. Observe that by uniqueness of the Kähler-Einstein potential $g^{E,\Omega_\infty}$ and by the theorem of Arzelà-Ascoli it is enough to prove that for every integer $k\in \mathbb{N}$ the sequence $\left(g^{(\nu)}\right)_{\nu \geq \nu_K}$ is bounded in $\mathcal{C}^k\left(K\right)$. Thus we are interested in obtaining $\mathcal{C}^k\left(K\right)$ estimates of the family $\left(g^{(\nu)}\right)_{\nu \geq \nu_K}$ of solutions to equation \mref{MAC}.
It relies on obtaining estimates of Sobolev norms of the sequences $\left(Log\left\lvert g^{E,\Omega_\nu}_{i\bar{j}} \right\rvert \right)_{\nu \in \mathbb{N}}$ and $\left(\Delta g^{E,\Omega_\nu} \right)_{\nu \in \mathbb{N}}$ on bounded subdomains of $\Omega_\infty$, where $\Delta$ denotes the Laplacian operator for the Euclidean metric on $\mathbb{C}^2$ (see the proof of Lemma 3 in \cite{Yeu} for more details).
\\Since $\Omega$ is a bounded convex domain with boundary of class $\mathcal{C}^\infty$, it follows from \cite{KZ} that there exists a number $0<a$ such that $D$ satisfies the $a$-squeezing property. Since $\Omega_\nu$ is biholomorphic to $D$, $\Omega_\nu$ also satisfies the $a$-squeezing property. From Proposition 3 in \cite{Yeu} we deduce that there exist constants $0<c\leq C$ such that for every integer $\nu\in\mathbb{N}$ we have $c\left[g^{K,\Omega_\nu}_{i\bar{j}}\right]\leq \left[g^{E,\Omega_\nu}_{\i\bar{j}}\right] \leq C\left[g^{K,\Omega_\nu}_{i\bar{j}}\right]$ on $\Omega_\nu$. Moreover recall that the sequence $\left(\left[g^{K,\Omega_\nu}_{i\bar{j}}\right]\right)_{\nu \in \mathbb{N}}$ converges uniformly on compact sets to $\left[g^{K,\Omega_\infty}_{i\bar{j}}\right]$. Therefore we obtain the uniform estimates by following line by line the proof of Lemma 3 in \cite{Yeu} (by replacing the balls $B_{\frac{a}{2}}(x), B_{a}(x)$ with bounded domains included in $\Omega_\infty$).
\end{proof}

\bibliographystyle{plain}
\bibliography{Biblio}

\end{document}